\documentclass[a4paper,12pt,oneside]{article}

\usepackage[latin1]{inputenc}
\usepackage[english]{babel}
\usepackage{amsthm}
\usepackage{amssymb}
\usepackage{amsmath}
\usepackage{hyperref}
\usepackage{mathrsfs}
\usepackage{soul}
\usepackage{pstricks}
\usepackage[nottoc]{tocbibind}

 \usepackage{tikz}

\usepackage{graphicx}
\graphicspath{{images/}}

\usepackage{longtable,geometry}
\usepackage{color}

\usepackage{hyperref}

\newtheoremstyle{note}{12pt}{12pt}{}{}{\bfseries}{.}{.5em}{}
\title{\LARGE\textbf{A Denjoy counterexample for circle maps with an half-critical point}}
\author{Liviana Palmisano\\ Institute of Mathematics of PAN\\
Warsaw 00-950, Poland}
\newtheorem{theo}[equation]{Theorem}
\newtheorem{prop}{Proposition}

\newtheorem{defin}[equation]{Definition}

\newtheorem{rem}[equation]{Remark}

\numberwithin{equation}{section}
\newtheorem{lemma}[equation]{Lemma}

\newcommand{\N}{{\mathbb N}}

\newcommand{\R}{{\mathbb R}}

\renewcommand{\S}{{\mathbb S}^1}

\newcommand{\Cinf}{{{\mathcal C}^\infty}}

\newcommand{\C}{{\mathcal C}}

\usepackage{makeidx}

\providecommand{\norm}[1]{{\lVert #1 \rVert}_{{\mathcal C}^n}}

\usepackage{hyperref}

\begin{document}
\maketitle
\author
\begin{abstract}
Modifying Hall's idea in \cite{hall} we construct an example of homeomorphism of the circle which is a Denjoy counterexample (\emph{i.e.} it is not conjugate to a rotation) and is a $\Cinf$-diffeomorphism everywhere except in a flat half-critical point.
\end{abstract}
\section{Introduction}
\subsection{Notations and Definitions}
Let us denote by $\S=\mathbb R/\mathbb Z$ the unit circle .
\paragraph{Rotation.}
For each $\rho\in\R$ let $R_{\rho}:\S\mapsto\S $ be the map defined by

\begin{equation*}
R_{\rho}(\theta+\mathbb Z) = (\theta+ \rho+\mathbb Z)
\end{equation*}

which is called the {\it rotation of angle $\rho$}.\\

\paragraph{Lifts.}

   Let $\pi:\R\mapsto\S$ be the projection of the real line to the circle and $f:\S\mapsto\S$ is continuous. We call a function $F:\R\mapsto\R$ a lift of $f$ if
\begin{equation*}
\pi\circ F=f\circ\pi.
\end{equation*}
In particular there exists the unique lift $F$ of $f$ such that $F(0)\in[0,1)$.\\

A lift $F$ inherits regularity properties of $f$ e.g. continuity, differentiability, smoothness.\\

\paragraph{Standing Notations.}
In the following we refer to maps on $\S$ with minuscule letters e.g $f$. The corresponding capital letter $F$ denotes the unique lift with the property $F(0)\in [0,1)$.
\par
Abusing notation we usually identify subsets of $\S$ with corresponding subsets of $\mathbb R$.

\begin{defin}
Let $f:\S\mapsto\S$ be a continuous function, we say that $f$ has degree one if for all $x\in\R$, $F(x+1)=F(x)+1$.
\end{defin}

\begin{defin}
A continuous function $f:\S\mapsto\S$ is non-decreasing\footnote{or orientation preserving} if its lift is non-decreasing .
\end{defin}

\subsection{Rotation number}
Let $f:\mathbb{S}^1\mapsto\mathbb{S}^1$ be a non-decreasing, continuous function of degree one then the limit
\begin{equation*}
\lim_{n\to\infty}\frac{F^{n}(x)}{n}
\end{equation*}
exists for every $x\in\R$ and it is independent of $x$.\\

This limit is called the \emph{rotation number} of $f$ and will be denoted by $\rho(f)$.\\

\begin{prop}
$\rho$ is non-decreasing: if $F_1< F_2$, then $\rho(f_1)\leq\rho(f_2)$.
\end{prop}

\begin{rem}\label{famrot1}
In particular, if ${(f_t)}_{t \in \R}$ is a family of continuous, non-decreasing functions  of degree one, such that for each $x\in \mathbb R $, $t\to F_{t}(x)$ is non-decreasing, then $t\to \rho(f_{t})$ is also non-decreasing.
\end{rem}
Moreover, at each irrational number the rotation number is increasing:
\begin{prop}
If $F_0< F_1$ and $\rho(f_0)$ is irrational, then $\rho(f_0)<\rho(f_1)$.
\end{prop}

\subsection{Flat half-critical points}
\label{def:demicritique}

We use the following notation for the left and right derivative of a function $f$ at $p\in\mathbb{S}^1$
\begin{eqnarray*}
	f_-'(p) = \lim_{h \to 0^-} \frac{f(p+h+\mathbb Z)-f(p)}{h}, \\
	f_+'(p) = \lim_{h \to 0^+} \frac{f(p+h+\mathbb Z)-f(p)}{h}.
\end{eqnarray*}
Iteratively we can define the left and right derivative of any order $n$.
\begin{defin}
Let $f:\S\mapsto\S$ be piecewise  $\C^{n+1}$. We say that $p\in \S$ is \emph{half-critical of order $n$} for $f$ if the left derivative of $f$ at $p$ is not zero and if all the right derivatives up to order $n$ are zero but not the derivative of order $n+1$.

A half-critical point of order $1$ is called shortly half-critical. \\
A point $p\in \S$ is said \emph{flat half-critical} for a piecewise $\Cinf$ function $f:\S\mapsto \S$ if the left derivative of $f$ at $p$ is not zero and the right derivatives of all order are zero.
\end{defin}

In this case, if $f$ is a function of the circle with a half-critical point $p$ of order $k \in \N$, piecewise $\C^n$, $\C^n$ on $\S \backslash \{p\}$, then, by slight abuse of  notation, we define:
\begin{displaymath}
\| f \|_{\C^n}=\left\|F\right\|_{\C^n}=\sup_{x\in [0,1] \atop 1\leq i\leq n}\left|\frac{d^iF}{dx^i}(x)\right|+\sup_{x\in [0,1]}\left|F(x)\right|.
\end{displaymath}

For point $p$ we abuse notation assuming that, for every $1\leq i\leq n$, $\frac{d^iF}{dx^i}(p)$ is the left derivative at $p$ of order $i$.

\begin{rem}
	We note that if $f:\S\mapsto \S$ is piecewise  $\C^{j}$ and $j \geq i$, then:
	$$ \left\|F\right\|_{C^j} \geq \left\|F\right\|_{C^i} $$

\label{rem:croissance}
\end{rem}

\subsection{Discussion and Statement of the Results}\label{DSR}
One of the main questions in the field of dynamical systems is whether a circle function is ``equivalent'' to a rotation. This means that if $f$ is a continuous function defined on the circle with rotation number $\rho$ and $R_{\rho}$ is the rotation of angle $\rho$, then there exists a continuous map $h:\mathbb S^1\mapsto\mathbb{S}^1$ such that $h\circ f=R_{\rho}\circ h$.
\par
If $h$ is strictly order preserving then we say that $f$ and $R_{\rho}$ are \emph{combinatorially equivalent} or \emph{semi-conjugate} and if $h$ is a homeomorphism then we say that $f$ and $R_{\rho}$ are \emph{topologically equivalent} or \emph{conjugate}.
\par
Poincar\'e in 1880 shows that any circle homeomorphism with irrational rotation number is combinatorially equivalent to a rotation. Denjoy in \cite{Denjoy} proves that any $\C^1$ diffeomorphism with irrational rotation number and with derivative of bounded variation is topologically equivalent to a rotation. Also in \cite{Denjoy}  Denjoy provides that the hypothesis on the
derivative is essential, in fact he gives examples of $\C^1$ diffeomorphisms with irrational rotation number which are not conjugate to a rotation. Since then examples of this kind, called Denjoy counterexamples, have been constructed by many mathematicians.
\par
In \cite{Katok} the author proves the existence of a Denjoy counterexample which is $\Cinf$ everywhere with the exception of one point which is a non-flat critical point for the function. In \cite{hall}, Hall constructs a $\Cinf$ Denjoy counterexample which has zero derivative in one or possibly two points which are flat critical points for the function.
\par
The idea of Hall is to construct the $\Cinf$ Denjoy counterexample as limit of successive perturbations of an initial continuous circle function with a flat interval. The idea is to remove in any step a portion of the flat interval containing the orbit of a initial fixed interval. This procedure generates in the limit a $\Cinf$ Denjoy counterexample with a flat critical point. If this process stops (in the case that the orbit of the initial fixed interval never enters in the flat interval) the author approximates his function with a $\Cinf$ circle endomorphism which has two critical points and without dense orbits (which is equivalent to be not conjugated to a rotation, see Lemma \ref{eq}).

\par
The first observation that we can do is that the latter case cannot happen. In fact the Denjoy counterexample produced in this case has critical points of bounded criticality and this is in contradiction with Yoccoz's result in \cite{Yo}. So Hall's construction always produces a $\Cinf$ Denjoy counterexample with only one flat critical point.
\par
In this paper we construct the Denjoy counterexample as in \cite{hall} as limit of successive perturbations of an initial circle function with a flat interval having the left boundary point of strictly positive left derivative. Modifying Hall's idea we are able to produce in the limit a Denjoy counterexample with only one flat half-critical point. More precisely:
\begin{theo}\label{conden}
Let $p$ be a point on the circle. For all irrational numbers $\rho \in[0,1)$ there exists a circle homeomorphism $f:\S\mapsto \S$ with rotation number $\rho$ which satisfies the following properties:

\begin{itemize}
	\item $f$ is piecewise $\Cinf$,
	\item $f$ is a $\Cinf$ diffeomorphism on  $\S \backslash \{p\}$,
	\item $p$ is a flat half-critical point for $f$,
	\item $f$ is a Denjoy counterexample.
\end{itemize}
\end{theo}
Our result complement the field of Denjoy counterexamples. It is also interesting because the function constructed is the degenerate case for upper circle maps. As such our research can be viewed as a step in developing the theory of such maps, a field which is very poorly explored with first results obtained only in \cite{Grfun}.

\section{Some Technical Lemmas}
\begin{lemma}\label{eq}
Let $f:\S\mapsto\S$  be a continuous, non-decreasing, degree one function with irrational rotation number $\rho\in\left[0,1\right)$. Then the following statements are equivalent:
\begin{enumerate}
\item $f$ is a Denjoy counterexample,
\item $f$ has not dense orbits,
\item $f$ has a wandering interval,\emph{ i.e.} there exists a non-empty set $I\subset\S$ such that, for all $n,m\in\N$, $n\neq m$, $f^n(I)\cap f^m(I)=\emptyset$,
\item there exists an interval $I\subset\S$ such that $\left|I\right|>0$ and $\left|f^n(I)\right|\rightarrow0$ for $n\rightarrow +\infty$.
 \end{enumerate}
\end{lemma}
The proof of this Lemma is known and it can be found for example in \cite{hall}, pag. 263.
\begin{lemma}
	Let ${f}  : \S\mapsto \S$ be a non-decreasing, degree one and  piecewise $\Cinf$ function, for which $\frac12$ is a flat half-critical point\footnote{See Definition \ref{def:demicritique}}.\\
	We suppose that there exists an interval $I$ of the form $I = (\frac12, c)$ with $c<1$, such that: $F'(x) = 0 \Leftrightarrow  x\in I$ and such that the left derivative $F_-'(\frac12) > K> 0$.\\
	Then, $\forall n \in \N$, $\forall \delta \in (0,1)$, and for each pair of intervals $I_1 = (\frac12, a)$ and $J_1 = (b,c)$ such that $I_1\cup J_1 \subset I$ and $I_1\cap J_1 = \emptyset$, there exists a non-decreasing, degree one and piecewise $\Cinf$ function,  $\tilde{f} = \tilde{f}_{n,\sigma, \delta} : \S \mapsto \S$, for which $\frac12$ is flat half-critical point. Moreover $\tilde{f}$ satisfy the following conditions:

	\begin{enumerate}
		\item $\norm{\tilde{F} - F} < \delta$,
		\item $\rho(\tilde{f} ) = \rho(f)$,
		\item $|\tilde{F'}(x) - F'(x)| < \delta F'(x)$, $\forall x \in (0,1) \backslash \overline{I}$,		
		\item $\tilde{F} '(x) = 0 \Leftrightarrow \forall x \in I_1\cup J_1$,
		\item $\tilde{F} '_-(\frac{1}{2}) >K> 0$.
	\end{enumerate}
\label{lemma:primo}
\end{lemma}

\begin{proof}
	Let $n \in \N$ and $0 < \delta < 1$.\\
	We observe that $I = (\frac12, c)$, $I_1 = (\frac12, a)$ and $J_1 = (b,c)$, with $\frac12 < a < b < c < 1$ (because $I_1 \cap J_1 = \emptyset$). The configurations of this intervals is clarified in the following figure. 

	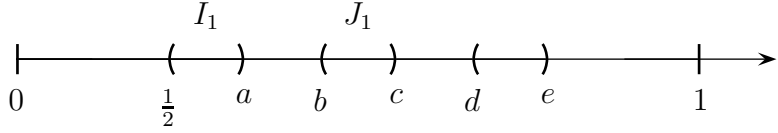
\begin{figure}[h]
	\centering
		\begin{pspicture}(0,-0.5)(10,0.6)
		\def\obj(#1)(#2)#3#4#5{%
		\psline[arrows={(-)}](#1)(#2)%
		\uput{0.4}[-90](#1){#4}\uput{0.4}[-90](#2){#5}%
		}

		\psset{tbarsize=2mm,arrowscale=2,linewidth=0.5pt}

		\psline{|-}(0,0)(1,0)\uput{0.4}[-90](0,0){0}
		\psline(2,0)\obj(2,0)(3,0){$I_1$}{$\frac12$}{$a$} \uput{0.4}[90](2.5,0){$I_1$}
		\psline(4,0)\obj(4,0)(5,0){$J_1$}{$b$}{$c$}       \uput{0.4}[90](4.5,0){$J_1$}
		\psline(6,0)\obj(6,0)(7,0){$J_1$}{$d$}{$e$}
		\psline{-|}(8,0)(9,0)\uput{0.4}[-90](9,0){1}
		\psline{->}(5,0)(10,0)
	\end{pspicture}

	\caption{Positions of the points as in Lemma \ref{lemma:primo}}

	\end{figure}

	\par
	Let $0 < \epsilon < \max(\frac12, 1-c)$, then by properties of $F$, the neighborhood of the flat interval $(\frac12 - \epsilon, c+\epsilon)$, has two connected components on $(0,1)$ on which $F'(x) > \xi(\epsilon) > 0$. So, we choose $(d,e) \subset (c+\epsilon, 1)$.
	\par
	We define for any $x \in [0,1]$,
	$$\tilde{F} (x) = F(x) + \frac{\delta}{3 \kappa c_n} \int_0^x \left(  \phi_{a,b}(t) - \phi_{d,e}(t) \right) dt $$
	where $\phi_{a,b}$ is a bump function supported on $[a,b]$\footnote{The same notation is used for the interval $[d,e]$.}, $\phi_{a,b} > 0$ on $(a,b)$. The constant $C_{d,e} > 0$ is first chosen such that $ \phi_{d,e} < \xi \delta $, and then, the constant $C_{a,b} > 0$ is chosen
 such that:\\
	$$ \int_0^1 \left(  \phi_{a,b}(t) - \phi_{d,e}(t) \right) dt = 0 .$$

	Finally, $\kappa = \max\left( (b-a)C_{a,b}, (e-d)C_{d,e} \right)$ and $c_n = 2 \max\left( \| \phi_{a,b} \|_{{\mathcal C}^n}, \| \phi_{d,e} \|_{{\mathcal C}^n} \right) \geq 1$.\\
	\par
	We observe that $\tilde{F}(0) = F(0)$ and that
	$$\tilde{F}(1) = F(1) + \frac{\delta}{3\kappa c_n}  \int_0^1 \left(  \phi_{a,b}(t) - \phi_{d,e}(t) \right) dt  = F(1) + \frac{\delta}{3\kappa c_n} \times 0 = F(1) = 1$$
	Moreover, $\tilde{F}$ can be extended on $\R$ (since $\tilde{F}(0) = F(0)$ and $\tilde{F}(1)=F(1)$), by:
	\[
		\tilde{F} (x) = \tilde{F} ( x - \lfloor x \rfloor ) + \lfloor x \rfloor
	\]
	(where $\lfloor x \rfloor$ denote the integer part of $x$),
	so $\tilde{F} $ is projected on a well defined function $\tilde{f}  : \S\mapsto \S$ which is of degree one. It is also easy to see that $\tilde{f}$ is piecewise $\Cinf$ and that $\frac12$ is a flat half-critical point for $\tilde{f}$.
	\par
	We prove that $\tilde{F} $ satisfies the properties $(1)$-$(5)$ above:
	\begin{enumerate}
	\item  For $n = 0$~:\\
	\begin{equation*}
		\left| F(x) - \tilde{F} (x) \right| = \frac{\delta}{3\kappa c_n} \left| \int_0^x   \phi_{a,b}(t) - \phi_{d,e}(t) dt \right| \leq  \frac{\delta}{3\kappa c_n} \kappa <  \delta.
	\end{equation*}
	Since the two functions  $\phi_{a,b}, \, \phi_{d,e}$  have disjoint supports, then the integral is bounded by $\kappa$.
	For all $1 \leq k \leq n$~:\\
	\begin{eqnarray*}
	\left| \frac{d^kF}{dx^k}(x) -  \frac{d^k\tilde{F} }{dx^k}(x) \right| = \left| \frac{\delta}{3\kappa c_n}  \frac{d^{k-1}}{dx^{k-1}} \frac{d}{dx}  \int_0^x  \phi_{a,b}(t) - \phi_{d,e}(t) dt \right| <\delta.
	\end{eqnarray*}
	\item
	Since the initial perturbation $\tilde{f}$ is $\mathcal C^0$ small, it can be majored by similarly small translations in either direction and for some intermediated translation the rotation number will be equal to $\rho(f)$. We observe that the property $\left\| F - \tilde{F} \right\|_{\C^0} < \delta$ remains true (the translation don't change the norm of superior order) and the other properties remain also verified.
	\item $\tilde{F} '(x) = F'(x) + \frac{\delta}{3\kappa c_n} \left( \phi_{a,b}(x) - \phi_{d,e}(x) \right)$. \\
We observe that on $(a,b)$ we have $\phi_{a,b}>0$ and $\phi_{d,e}=0 $, so $\tilde{F}$ is non-decreasing on $(a,b)$. \\
Moreover on $(d,e)$ the function $\phi_{a,b}=0$ and $\phi_{d,e} > 0$. We also have that $\phi_{d,e} < \delta \xi$ and $F' > \xi$\footnote{This comes from the fact that $(d,e) \subset (c,1)$.}; then the fact that $F' - \phi_{d,e} > 0$  implies that $\tilde{F}$ is non-decreasing on $(d,e)$. \\
On the other intervals, we have $\tilde{F}' = F'$. This fact proves the point $(3)$ and the fact that the function is non-decreasing. Finally, we have proved the points $(3)$ and $(4)$.
	\item $\tilde{F} '_-(x) = F'_-(x)$ on $(0,a) \ni \frac12$, and $F'_-(\frac12) > K>0$.
	\end{enumerate}

\end{proof}
\begin{lemma}\label{lemma:secondo}
	Let $\tilde{f}  : \S\mapsto \S$ be a non-decreasing, degree one, piecewise $\Cinf$ function, for which $\frac12$ is a flat half-critical point.\\
	We assume that $\tilde{f} $ satisfies the following properties:
	\begin{itemize}
		\item There exist two intervals $I$ and $J$, $I = (\frac12, a)$, $J = (a',b)$, $a' > a$, $b < 1$, such that:
	$$ x \in I\cup J \Leftrightarrow \frac{d\tilde{F} }{dx}(x) = 0 $$
	and the left derivative of $\tilde{F} $ in $\frac12$ is $\tilde{F} _-'(\frac12) > K>0$.
	\end{itemize}
	Then, $\forall n \in \N$, $\forall \sigma \in (0,1)$, $\exists g = g_{n,\sigma} : \S \mapsto \S$ a non-decreasing, degree one, piecewise $\Cinf$ function for which $\frac12$ is a flat half-critical point. Moreover $g$ satisfies the following conditions:
	\begin{enumerate}
		\item $\norm{ G - \tilde{F}  } <  \sigma$,
		\item $\rho(g) = \rho(\tilde{f} )$,
		\item $| G'(x) - \tilde{F'}(x)|  <  \sigma \tilde{F'}(x)$, $\forall x \in (0,1) \backslash \overline{(\frac12,b)}$,
		\item $G'(x) = 0 \Leftrightarrow x\in I$,
		\item $G'_- (\frac12) >K> 0$,
                     \item  on some right-sided neighborhood of $a$, $g$ can be represented as $
h_{r,{n+1}}((x-a)^{n+2})$
where $h_{r}$ is a $\Cinf$-diffeomorphism on an open neighborhood of $a$.
		
	\end{enumerate}
\end{lemma}

\begin{proof}
	This proof is really similar to the proof of Lemma \ref{lemma:primo}.\\
	Let $n \in \N$ an integer, $0< \sigma <1$ a fixed real number and $(c,d) \subset (b,1)$.\\
	We have the following configuration,
	\begin{figure}[h]
	\centering
		\begin{pspicture}(0,-0.5)(10,0.6)
		\def\obj(#1)(#2)#3#4#5{%
		\psline[arrows={(-)}](#1)(#2)%
		\uput{0.4}[-90](#1){#4}\uput{0.4}[-90](#2){#5}%
		}

		\psset{tbarsize=2mm,arrowscale=2,linewidth=0.5pt}

		\psline{|-}(0,0)(1,0)\uput{0.4}[-90](0,0){0}
		\psline(3,0)\obj(3,0)(4,0){$I$}{$\frac12$}{$a$} \uput{0.4}[90](3.5,0){$I$}
		\psline(5,0)\obj(5,0)(6,0){$J$}{$a'$}{$b$}       \uput{0.4}[90](5.5,0){$J$}
		\psline(7,0)\obj(7,0)(8,0){$J$}{$c$}{$d$}
		\psline{-|}(8,0)(9,0)\uput{0.4}[-90](9,0){1}
		\psline{->}(5,0)(10,0)
	\end{pspicture}

	\caption{Positions des points du lemme \ref{lemma:primo}}

	\end{figure}
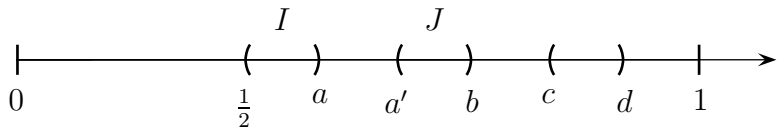

	We use the bump function $\phi_{\cdot,\cdot}$ and its properties.\\
	In this Lemma, $(c,d)$ plays the same role as $(d,e)$ in Lemma $\ref{lemma:primo}$, the constants $\kappa$ and $c_n$ depend  on $(a,b)$ and $(c,d)$.\\
	So, on $(c,d)$ we have $\tilde{F}' > \xi$ and we choose $C_{c,d}$ such that:
	$$ \phi_{c,d} < \sigma \xi.$$
	This condition guarantees that the constructed function is non-decreasing (see the proof of Lemma \ref{lemma:primo}).\\
	As before, we choose $C_{a,b}$ in a such way that:
	$$ \int_0^1 \left(  \phi_{a,b}(t) - \phi_{c,d}(t) \right) dt = 0.$$
	We denote:
	\[
		G (x) = \tilde{F} (x) + \frac{\sigma}{3\kappa c_n} \int_0^x \left( \phi_{a,b}(t) - \phi_{c,d}(t) \right) dt.
	\]
	\par
          We observe that $G$ could be extended on $\R$ (because $G(0) = \tilde{F} (0)$ and $G(1)=\tilde{F} (1)$), by:
	\[
		 G(x) = G( x - \lfloor x \rfloor ) + \lfloor x \rfloor
	\]
	so it is projected on a well defined, degree one function $g : \S \mapsto \S$ which is piecewise $\Cinf$ and for which $\frac12$ is a flat half-critical point.
	\\
	\par
	The proof of the fact that $G$ satisfies the properties $(1)$-$(5)$ is exactly the same that the proof of the Lemma \ref{lemma:primo} for $\tilde{F}$. In order to ensure $(6)$ we observe that on a right-sided neighborhood of $a$ we can easily change the flatness of the bump function $\phi_{\cdot,\cdot}$ so it fulfills the required condition.

\end{proof}

\section{Proof of Theorem \ref{conden} }

Let $\rho\in\left[0,1\right)$ be a fixed irrational number. We chose $p = \frac12$\footnote{Without loosing generality, the proof remains the same for any $p \in \S$, changing $\frac12$ by $p$.}, which will be the flat half-critical point.

The Denjoy counterexample which we will construct  will be defined as the limit of a sequences of functions:
\begin{displaymath}
(f_n:\S\mapsto\S)_{n \in \N},
\end{displaymath}
for which there exist:
\begin{itemize}
\item an interval $J_0$ and a sequence of intervals:
\begin{displaymath}
(I_n)_{n \in \N}
\end{displaymath}
such that $$\bigcap_{i\geq 0}\overline{I_i}=\left\{\frac{1}{2}\right\},$$
\item a sequence of integers
\begin{displaymath}
1=r_0< r_1 < r_2 < \cdots< r_n
\end{displaymath}
\end{itemize}
 such that, for all $i\in\{0,1,2,\ldots,n\}$ the following conditions are satisfied:
\begin{enumerate}
\item $f_i$ is a piecewise $\Cinf$, non-decreasing, degree one map,
\item $\rho(f_i)=\rho$,
\item ${F'_i}(x)=0$ if and only if $x\in I_i$,
\item the left derivative of $f_i$, $(F_{i})'_-\left(\frac{1}{2}\right)>K>0$,
\item if $b_i$ denote the right boundary point of $I_i$, then on some right-sided neighborhood of $b_i$, $f_i$ can be represented as $
h_{r,i}((x-b_i)^{i+1})$
where $h_{r,i}$ is a $\Cinf$-diffeomorphism on an open neighborhood of $b_i$,
\item $|I_i|<\frac{1}{2} |I_{i-1}|$,
\item $\left\|F_i-F_{i-1}\right\|_{C^{i-1}}<\frac{1}{2^{i}}$,
\item $0<\left|f_i^j\left(J_0\right)\right|<\frac{1}{2^{k-1}}$, if $r_{k-1}\leq j<r_k$ for $k \in \{1,2,\ldots,i\}$,
\item $f_i^j\left(J_0\right)\cap I_i=\emptyset$, if $0\leq j<r_i$ and $f_{i}^{r_i}\left(J_0\right)\Subset I_i$,
\item $| F_i'(x) - F_{i+1}'(x) | < F_i'(x)\frac{1}{2^{i+1}}$, for all $x \in (0,1) \backslash \overline{I_i}$.

\end{enumerate}

\par
We prove the existence of such a sequence by induction.

\subsubsection{Initialization}
Construction of $f_0$.\\

We start to construct a piecewise $\Cinf$, non-decreasing, degree one map $\tilde{f}$ with the following properties:
\begin{itemize}
\item[-] $\tilde{F}'(x)=0$ if and only if $x\in(\frac12,\frac34]$,
\item[-] The left derivative of $\tilde{f}$, $(\tilde{F})_-'(\frac12)>K>0$,
\item[-] on some right-sided neighborhood of $\frac34$, $\tilde{f}$ can be represented as $
h_{r,1}((x-\frac34)^{2})$
where $h_{r,1}$ is a $\Cinf$-diffeomorphism on an open neighborhood of $\frac34$.
\end{itemize}
Now, let $f_0 = \tilde{f} + t(\rho)$, where $t$ is a real number depending on $\rho$ in the way that the rotation number of $f_0$ is $\rho$, see Remark \ref{famrot1}.\\
We denote $I_0 = (\frac12,\frac34]$. Since $f_0$ is a diffeomorphism on $[0,1] \backslash\overline{I_0}$ in its image, we take $J_0$ as any subinterval of $f_0^{-1} ( I_0 )$.\\

So $f_0$ is a piecewise $\Cinf$, non-decreasing, degree one map and it has rotation number $\rho$.
\par
The conditions $(1)$-$(10)$ are also satisfied by $f_0$.\\
\par
\subsubsection{Induction}

We assume now that $f_n$ is constructed. We construct $f_{n+1}$ by perturbing $f_n$ in the way that the conditions $(1)$-$(10)$ are still satisfied.
\par
We divide the interval $I_n$ into two subintervals $I_{n+1}$ and $J_{n+1}$ such that:
\begin{itemize}
\item $I_{n+1}$ is of the form $I_{n+1} = (\frac12,\cdot)$,
\item $I_{n+1}\cup J_{n+1} \subset I_n$,
\item $I_{n+1}\cap J_{n+1}=\emptyset$,
\item $f_n^{r_n}(J_0)\subset J_{n+1} $.
\end{itemize}

\vspace{.5cm}
First step\\
\par

We apply Lemma \ref{lemma:primo} and we get a non-decreasing, degree one and piecewise $\Cinf$ function $\tilde{f}_{n,\delta}$, which satisfies to the following conditions:
\begin{itemize}
\item$\norm{\tilde{F}_{n,\delta} - F_n} < \frac{\delta}{2^{n+2}}$,
\item$|\tilde{F'}_{n,\delta} - F'_n| < F_n'(x) \frac{\delta}{2^{n+1}}$, $\forall x \in(0,1)\setminus \overline{I_n}$,
\item $\rho(\tilde{f}_{n,\delta}) = \rho(f_n)$,
\item $\tilde{F}_{n,\delta'}(x) = 0 \Leftrightarrow x \in I_{n+1} \cup J_{n+1}$,
\item the left derivative $\tilde{F}_{n,\delta_-}'(\frac{1}{2})>K>0$.
\end{itemize}
Since, by construction, $\tilde{F}_{n,\delta}^i\rightarrow F_n^i$ for $\delta\rightarrow 0$ uniformly for $i\in\{1,2,\ldots,r_n\}$, then we can fix $\delta'<\delta<1$, such that:
\begin{displaymath}
\left|\tilde{f}_{n,\delta'}^j\left(J_0\right)\right|<\frac{1}{2^{k-1}}\textrm{ for all } r_{k-1}\leq j <r_k, \textrm{  } k\in\{1,2,\ldots,n\},
\end{displaymath}

\begin{equation}\label{paura1}
\tilde{f}_{n,\delta'}^j\left(J_0\right)\cap I_n=\emptyset\textrm{ for all }0\leq j<r_n
\end{equation}
and
\begin{equation}\label{paura}
\tilde{f}_{n,\delta'}^{r_n}\left(J_0\right)\subset J_{n+1}.
\end{equation}

We study now the orbit of $J_{n+1}$ under $\tilde{f}_{n,\delta'}$.\\
We could have two different cases:

\begin{itemize}
\item[$\ast$] there exists $m>0$ such that $\tilde{f}_{n,\delta'}^m(J_{n+1}) \in I_n$,
\item[$\ast$] for all $m>0$, $\tilde{f}_{n,\delta'}^m(J_{n+1}) \notin I_n$.
\end{itemize}

Observe that the second situation never occurs. We can construct $g:\S\mapsto \S$ being non-decreasing function of degree one, which is equal to $\tilde{f}_{n,\delta'}$ everywhere except on $I_n\setminus J_{n+1}$. Moreover, $g$ can be chosen so that to the left of the right boundary point of $I_n\setminus J_{n+1}$, $a_{n}$ it is equal to $h_{l,n}((a_{n}-x)^{{n+1}})$ for some $\Cinf$-diffeomorphism $h_{l,{n}}$ such that $h_{l,n}'(0) \neq 0$. Such $g$ belongs to the class of functions with a flat interval (being $J_{n+1}$ in this case) studied in \cite{a} (see property (5)). We notice that $g^m(J_{n+1}) = \tilde{f}_{n,\delta'}^m(J_{n+1})$ thus the orbits of $g$ are not dense. By Lemma \ref{eq} it has a wandering interval and therefore contradicts Corollary to Theorem $1$ in \cite{a}.



We study now the first case: we may assume that $m$ is the smallest positive integer such that $\tilde{f}_{n,\delta'}^m(J_{n+1})\in I_n$. We shall prove that there exists $\delta''$ (smaller than $\delta'$ if necessary) such that $\tilde{f}_{n,\delta''}^m(J_{n+1})$ is in the interior of $ I_{n+1}$.

The functions $\tilde{f}_{n,\delta}$ approximate $f_n$, in particular
\[
	      \tilde{f}_{n,\delta}^m(J_{n+1}) \to f_n^m(I_n),
\]
as $\delta\searrow 0$. Moreover $\delta \mapsto \tilde{f}_{n,\delta}^m(J_{n+1})$ is continuous, consequently the set
\[
	A=\{\tilde{f}_{n,\delta}^m(J_{n+1}):\delta \in [0,\delta']\},
\]
is an interval (of $\S$) containing $\tilde{f}_{n,\delta}^m(J_{n+1})$ and $f_n^m(I_n)$. The proof is concluded once we observe that, because of the fact that the rotation number $\rho$ is irrational, $A \cap J_{n+1} = \emptyset$, so $A$ covers a portion of the interior of $I_{n+1}$.\\

Second step.\\

By Lemma \ref{lemma:secondo} for all $\sigma > 0$, there exists $f_{n+1,\sigma}:\S\mapsto\S$, piecewise $\Cinf$, non-decreasing, degree one map such that:
\begin{itemize}
\item $\norm{ F_{n+1,\sigma} - \tilde{F}_n } <  \frac{\sigma}{2^{n+2}}$,
\item $|\tilde{F'}_{n+1,\sigma} - F'_n| < F_n'(x) \frac{\sigma}{2^{n+1}}$, $\forall x \in(0,1)\setminus\overline{I_{n+1}}$,
\item $\rho(f_{n+1,\sigma})=\rho$,
\item $F_{n+1,\sigma}'(x)=0$ if and only if $x\in I_{n+1}$,
\item $\left(F_{n+1,\sigma}\right)'_- (\frac12) > K> 0$,
\item if $b_{n+1}$ denote the right boundary point of $I_{n+1}$, then on some right-sided neighborhood of $b_{n+1}$, $f_{n+1}$ can be represented as $
h_{r,{n+1}}((x-b_{n+1})^{n+2})$
where $h_{r,{n+1}}$ is a $\Cinf$-diffeomorphism on an open neighborhood of $b_{n+1}$.
\end{itemize}

Observe that, since $f^m_{n+1,\sigma}\to \tilde{f}^m_n$ for $\sigma\to 0$
\begin{equation}\label{eq:sei}
f_{n+1,\sigma}^{m} (J_{n+1}) \Subset I_{n+1}
\end{equation}

Finally we denote $r_{n+1}=r_n+m$. Since $f_{n+1,\sigma}^i\rightarrow \tilde{f}_{n}^i$ uniformly for all $i\in\{1,2,\ldots,r_n+m\}$ and since $\tilde{f}_{n}^i\left(J_{0}\right)$ is a singleton for $i>r_n$, then  we can affirm that (we set $f_{n+1}=f_{n+1,\sigma}$):
\begin{displaymath}
\left|f_{n+1}^j\left(J_{0}\right)\right|<\frac{1}{2^{k-1}}
\end{displaymath}
for all $r_{k-1}\leq j<r_k$, with $k\in\{1,2,\ldots,n\}$, and
\begin{displaymath}
\left|f_{n+1}^j\left(J_{0}\right)\right|<\frac{1}{2^n}
\end{displaymath}
if $r_n\leq j< r_{n+1}$.\\
By (\ref{eq:sei}), (\ref{paura1}),  (\ref{paura}) and since $f_{n+1,\sigma}^i\rightarrow \tilde{f}_n^i$ uniformly for all $i\in\{1,2,\ldots,r_{n+1}\}$, then
\begin{displaymath}
f_{n+1}^i\left(J_{0}\right)\cap I_{n+1}=\emptyset\textrm{ for all }0\leq i<r_{n+1}
\end{displaymath}
and

\begin{displaymath}
f_{n+1}^{r_{n+1}}\left(J_{0}\right) \Subset I_{n+1}.
\end{displaymath}

\par
So we have constructed a sequence of functions $(f_n)_{n \in \N}$ satisfying conditions $(1)$-$(10)$ for all $n\geq 1$. \\
We show that for $n \in \N$, the sequence $(f_k)$ converges in the sense of the norm $\|\cdot\|_{\C^n}$ (see Remark \ref{def:demicritique}).
The sequence $(f_k)$ is a Cauchy sequence. In fact, let $n \in \N$ fixed. By point $(7)$, for all $i\in\N^\star$,
$$ \left\|F_i-F_{i-1}\right\|_{C^{i-1}}<\frac{1}{2^{i}} $$
So by Remark \ref{rem:croissance},
$$ \left\|F_i-F_{i-1}\right\|_{C^n} \leq \left\|F_i-F_{i-1}\right\|_{C^{i-1}}<\frac{1}{2^{i}}, \text{ for all } i > n $$
where, for $i > n$, fixed, and for all $p > q > n$ we have:
$$ \left\|F_p-F_q\right\|_{C^n}<\sum_{k=q-1}^{p-1}\frac{1}{2^{k}} < \sum_{k=q-1}^{+\infty}\frac{1}{2^{k}} = \frac{1}{2^{q-2}}$$

So, for all $n \in \N$, the sequence $(f_k)$ converges in the sense of the norm $\| \cdot \|_{\C^n}$ to a piecewise $\C^n$, non-decreasing, degree one circle map $f$ which has rotation number $\rho$  (see points $(1)$ and $(2)$).
\par

By point $(4)$ the left derivative of $f$ at $\frac12$ is not zero and by point $(3)$ the right derivatives of all orders are zero.

So, $\frac12$ is a flat half-critical point for $f$.

The conditions $(3)$ and $(10)$ guarantee us that $f$ has not other critical points on $\S \backslash \{ \frac12 \}$.

\par
To conclude, by conditions $(8)$ and $(9)$,
\begin{displaymath}
\left|f_n^i\left(J_0\right)\right|\rightarrow 0\textrm{ for }i\rightarrow+\infty
\end{displaymath}
uniformly in $n$, and then
\begin{displaymath}
\left|f^i\left(J_0\right)\right|\rightarrow 0\textrm{ if }i\rightarrow+\infty.
\end{displaymath}
By Lemma \ref{eq} $f$ has a wandering interval and it is not conjugated to a rotation.

\subsection*{Acknowledgments}
I would sincerely thank Prof. J. Graczyk for introducing me to the initial subject of this paper, for his valuable advice on its drafting and for his continuous encouragement. I am also very grateful to the referee for his constructive comments.The paper was partially supported by funds allocated to the implementation of the international co-funded project in the years 2014-2018, 3038/7.PR/2014/2, and by the EU grant PCOFUND-GA-2012-600415.

\end{document}